\documentclass[10pt]{article}
\usepackage[latin1]{inputenc}
\usepackage{amsmath, amssymb, amsthm}
\usepackage[dvips]{graphicx}
\usepackage{mathrsfs}
\numberwithin{equation}{section}

\newcommand{\var}{\operatorname{Var}}

\newcommand{\bvar}{\operatorname{\mathbb{V}ar}}

\newcommand{\id}{\mbox{\rm1\hspace{-.2ex}\rule{.1ex}{1.44ex}}
   \hspace{-.82ex}\rule[-.01ex]{1.07ex}{.1ex}\hspace{.2ex}}

\newcommand{\barre}[1]{\overline{#1}}
\newtheorem{theo}{Theorem}[section]
\newtheorem{pr}{Proposition}[section]

\newtheorem{lem}{Lemma}[section]

\newtheorem{rmk}{Remark}[section]

\begin{document}

\title{Quenched Central Limit Theorems for Random Walks in  Random Scenery}
\author{Nadine Guillotin-Plantard \footnotemark[1] \ \ and Julien Poisat \footnotemark[2] }
\date{}

\footnotetext[1]{Institut Camille Jordan, CNRS UMR 5208, Universit\'e de Lyon, Universit\'e Lyon 1, 43, Boulevard du 11 novembre 1918, 69622 Villeurbanne, France. E-mail: nadine.guillotin@univ-lyon1.fr}
\footnotetext[2]{Mathematical Institute, Leiden University, P.O. Box 9512, NL-2300 RA Leiden, The Netherlands. E-mail: poisatj@math.leidenuniv.nl
This work was supported by the french ANR project MEMEMO2 10--BLAN--0125--03.\\
}

\maketitle

\begin{abstract}
Random walks in random scenery are processes defined by  
$$Z_n:=\sum_{k=1}^n\omega_{S_k}$$ where $S:=(S_k,k\ge 0)$ is a random walk evolving in $\mathbb{Z}^d$ and $\omega:=(\omega_x, x\in{\mathbb Z}^d)$ is a sequence of i.i.d. real random variables.
Under suitable assumptions on the random walk $S$ and the random scenery $\omega$,  almost surely with respect to $\omega$, the correctly renormalized sequence $(Z_n)_{n\geq 1}$ is proved to converge in distribution to a centered Gaussian law  with explicit variance.
 \end{abstract}

\section{Introduction and results}

\paragraph{The random walk in random scenery.}Let $\omega= (\omega_x)_{x\in \mathbb{Z}^d}$ be a sequence of IID real random variables (the scenery) defined on a probability space $(\Omega,\mathcal{F},\mathbb{P})$.
Let $(X_n)_{n\geq1}$ be a sequence of IID random variables defined on another probability space $(\Omega',\mathcal{A},P)$, taking values in $\mathbb{Z}^d$, $d\geq 1$. Define the random walk $S=(S_n)_{n\geq 0}$ by $S_0 = 0$ and
\begin{equation*}
 \forall n\geq 1, \quad S_n = \sum_{k=1}^n X_k.
\end{equation*}
When the support of $X_1$ is a subset of $\mathbb{N}^*$, $(S_n)_{n\geq 0}$ is called a renewal process.

\textit{Each time the random walk is said to evolve in $\mathbb{Z}^d$, it implies that the walk is truly 
$d$-dimensional, i.e. the linear space generated by the elements in the support of $X_1$ is $d$-dimensional.}

The random walk in random scenery (RWRS) is the process defined by
\begin{equation*}
\forall n\geq 1, \quad Z_n = \sum_{k=1}^n \omega_{S_k}.
\end{equation*}
In other words, $Z_n$ is the sum of the random variables $(\omega_x)_{x\in\mathbb{Z}^d}$ collected by the random walk $S$ up to time $n$.

RWRS was first introduced in dimension one by Kesten and Spitzer \cite{MR550121} and Borodin \cite{MR543530,MR535455} in order to construct new self-similar stochastic processes. Functional limit theorems for RWRS have been first obtained under the product measure $\mathbb{P}\otimes P$, usually called the {\it annealed} case. For $d=1$, Kesten and Spitzer \cite{MR550121} proved that when $X$ and
$\omega$ belong to the domains of attraction of different stable laws
of indices $1<\alpha\leq 2$ and $0<\beta\leq 2$, respectively,
then there exists $\delta>\frac{1}{2}$ such that
$\big(n^{-\delta}Z_{[nt]}\big)_{t\geq 0}$ converges weakly as $n\rightarrow
\infty$ to a continouous $\delta$-self-similar process with stationary increments,
$\delta$ being related to $\alpha$ and $\beta$
 by $\delta=1-\alpha^{-1}+(\alpha\beta)^{-1}$. The limiting process can be seen as a mixture of  $\beta$-stable processes, but it is not a stable process.
When $0<\alpha<1$ and for arbitrary $\beta$, the sequence $\big(n^{-\frac{1}{\beta}}Z_{[nt]}\big)_{t\geq 0}$ converges
weakly, as $n\rightarrow \infty$, to a stable process with index $\beta$ (see \cite{CGPP}). Bolthausen \cite{MR972774} (see also \cite{MR2883216}) gave a method to solve the case $\alpha=1$ and $\beta=2$ and especially, he proved
that when $(S_{n})_{n\in \mathbb{N}}$ is a recurrent $\mathbb{Z}^{2}$-random
walk, the sequence $\big((n\log n)^{-\frac{1}{2}}Z_{[nt]}\big)_{t\geq 0}$ satisfies a functional
central limit theorem. More recently, the case $d=\alpha\in\{1,2\}$ and $\beta\in (0,2)$ was solved in \cite{CGPP},  the authors prove that the sequence 
$\big(n^{-1/\beta} (\log n)^{1/\beta -1}   Z_{[nt]} \big)_{t\geq 0}$ converges weakly  to a stable process with index $\beta$. Finally for any arbitrary transient $\mathbb{Z}^{d}$-random walk, it can be shown that the sequence $(n^{-\frac{1}{2}}Z_{n})_n$ is asymptotically normal (see for instance \cite{MR0388547} page 53). 

Far from being exhaustive, we can cite strong approximation results and laws of the iterated logarithm \cite{MR688990,MR1700010,MR1624017}, limit theorems for correlated sceneries or walks \cite{MR2571841,MR2744890}, large and moderate deviations results \cite{MR2288063,MR1840830, MR2060457, MR2288269}, ergodic and mixing properties (see the survey \cite{MR2306188}). 

Our contribution in this paper is a {\it quenched scenery} version of the distributional limit theorems, i.e. we prove that for $\mathbb{P}$-almost every fixed path of the random scenery, a limit theorem holds for $Z_n$ (correctly renormalized).
It is worth remarking that when the random walk $S$ is fixed, functional limit theorems for the sequence  $(Z_{[nt]})_{t\geq 0}$ have been proved (see \cite{MR2883216,MR2030745,MR2779486}). Indeed, conditionally to the random walk, the sum  $Z_n$ can be viewed as a sum of IID random variables weighted by the local time of the random walk. Roughly speaking functional central limit theorems hold true as soon as the self-intersection local time of the random walk converges almost surely to some constant. 
To be complete let us mention that in the case when the scenery is given as a sequence of positive and heavy-tailed IID random variables coarse graining techniques have been used in \cite{MR2353391} to derive a distributional limit theorem for $\mathbb{P}$-almost every realization of the scenery when the random walk evolves in $\mathbb{Z}^d$ with $d\geq 2$. However, the  coarse graining scheme is adapted to heavy-tailed
environment, which is quite different from our setup (in which scenery
random variables have at least finite second moment), since it relies on
the existence of traps, that is, roughly speaking, regions of the
environment with large values. Therefore, it does not seem to be
applicable to our context, although another coarse graining technique
could presumably be designed to tackle the case of $d=2$ and scenery
random variables with finite second moment. 

\paragraph{Statement of the results.}  We denote by $\mbox{\bf (A)}$ the following assumptions on the random scenery $(w_{x})_{x\in\mathbb{Z}^d}$:
\begin{align}
&\omega_0 \stackrel{\rm{(law)}}{=} -\omega_0,\label{asspt1}\\ 
&\mathbb{E}(|\omega_0|^k) < +\infty  \quad(\forall k\geq 1), \label{asspt2}\\
&\mathbb{E}(\omega_0^2) = 1.\label{asspt4}
\end{align}
Quenched central limit theorems are proved for the RWRS in the three following cases:
 \begin{theo}[Renewal process case]\label{quenched_clt}
Suppose that $d=1$, that the support of $X_1$ is an aperiodic set of $\mathbb{N}^*$, $m := E[X_1]<+\infty$ and $E[X_1^2]<+\infty$. Then, under assumption $\mbox{\bf (A)}$, we have $\omega$-a.s,
 \begin{equation*}
 ([Z_n - EZ_n]/\sqrt{n})_{n\geq1} \stackrel{\rm{(law)}}{\longrightarrow} \mathcal{N}(0,1-1/m).
 \end{equation*}
\end{theo}
In the case of planar random walks with finite non-singular covariance matrix, we are able to prove a convergence in law to a Gaussian random variable along some subsequences:
\begin{theo}\label{quenched_clt2}
 Suppose that $d=2$, that the random walk is symmetric with finite non-singular covariance matrix $\Sigma$. Then, under assumption $\mbox{\bf (A)}$, for all $\nu>0$, $\omega$-a.s 
 \begin{equation*}
  \frac{Z_{t_m}}{\sqrt{t_m \log t_m}} \stackrel{\rm{(law)}}{\longrightarrow} \mathcal{N}(0,\sigma^2),
 \end{equation*}
where $t_m = [\exp(m^{1+\nu})]$ and $\sigma^2 = (2\pi\sqrt{\det \Sigma})^{-1}$.
\end{theo}

\begin{theo}\label{quenched_clt3}
Suppose that the random scenery is centered, square integrable with variance one and that one of the following assumptions 
holds:
\begin{description}
 \item[{\bf Assumption  (1) }:] $d\geq3$ and $X_1$ is square integrable. 
 \item[Assumption (2) :] $X_1$ is in the domain of attraction of a strictly stable law with index $\alpha\in(0,2)$ and $d>\alpha$.
\end{description}
Then, we have $\omega$-a.s,
 \begin{equation*}
 (Z_{[nt]}/\sqrt{n})_{t\geq 0} \stackrel{\rm{(law)}}{\longrightarrow}  (\sigma B_t)_{t\geq 0},
 \end{equation*}
 where $(B_t)_{t\geq 0}$ is a real Brownian motion and
 \begin{equation*}
 \sigma^2 = \gamma^2 \sum_{k\geq 0} k^2 (1-\gamma)^{k-1}
 \end{equation*}
 with
 \begin{equation*}
  \gamma = P(S_k \neq 0\, \mbox{\rm for any\ } k\geq 1)\in (0,1).
 \end{equation*}
\end{theo}
\noindent\mbox{\bf Remarks:}\\*
{\bf 1-} Assumptions in Theorems \ref{quenched_clt} and \ref{quenched_clt3} imply that the random walk is transient.
It easily follows from Local limit theorems for the random walk (see \cite[Theorem 1]{MR0222939}).\\*
{\bf 2-} Theorem \ref{quenched_clt3} still holds if the random variables $X_n, n\geq 1$ are only assumed to belong to the basin of attraction of a stable distribution. \\*

Our paper is organized as follows. In Section 2, results about self-intersections and mutual intersections of $\mathbb{Z}^d$- random walks with finite or infinite second moment are collected. In Section 3, Theorems \ref{quenched_clt}, \ref{quenched_clt2} and \ref{quenched_clt3} are proved. In order to prove Theorem \ref{quenched_clt3}, we see the process $(Z_n)_{n\geq0}$ as a functional of the {\it environment viewed from the particle} (i.e. from the random walker) and use a theorem of Derrienic and Lin \cite{DL} on invariance principles for Markov chains started at a point to deduce a quenched invariance principle for $(Z_n)_{n\geq 0}$. This approach can be followed as soon as the number of mutual intersections of two independent copies of the random walk up to time $n$ is asymptotically less than $\sqrt{n}$. Since this does not apply under the assumptions of Theorems \ref{quenched_clt} and \ref{quenched_clt2}, we rather use for those cases the method of moments, at the cost of assuming moments of any order for the scenery.

\section{Preliminaries}

In the following, we use the notation $$\barre{X} := X - EX$$ and $$u_n = \mathcal{O}(v_n) \Longleftrightarrow \exists C>0 : |u_n| \leq C v_n.$$ 

\paragraph{Mutual and self-intersection local times.} We will denote by $N_n(x)$ the time spent by the walk in site $x$ up to time $n$, that is
\begin{equation*}
 \forall x\in \mathbb{Z}^d,\, \forall n\geq 1,\quad N_n(x) = \sum_{k=1}^n \id_{\{S_k = x\}}. 
\end{equation*}
The RWRS can be rewritten as
\begin{equation*}
 Z_n = \sum_{x\in\mathbb{Z}^d} \omega_x N_n(x),
\end{equation*}
and we have
\begin{equation*}
 \overline{Z}_n = \sum_{x\in\mathbb{Z}^d} \omega_x \overline{N_n(x)}.
\end{equation*}
We will denote by $I_n^{[p]}$ the $p$-fold self-intersection local time of the walk up to time $n$, which writes
\begin{equation*}
 \forall p\geq 2, \forall n\geq1,\quad I_n^{[p]} = \sum_{x\in\mathbb{Z}^d} [N_n(x)]^p= \sum_{1\leq k_1,\ldots, k_p \leq n} \id_{\{S_{k_1}=S_{k_2}=\ldots=S_{k_p}\}},
\end{equation*}
and we abbreviate $I_n^{[2]}$ as $I_n$. If $N^{(1)}_n(x)$ and $N^{(2)}_n(x)$ are the respective local times of $S^{(1)}$ and $S^{(2)}$, two independent copies (or replicas) of the walk $S$ -- these notations will be kept throughout the paper -- we define $Q_n^{[p,q]}$ by
\begin{equation*}
Q_n^{[p,q]} = \sum_{x\in\mathbb{Z}^d} [N^{(1)}_n(x)]^p [N^{(2)}_n(x)]^q = \sum_{\substack{1\leq k_1,\ldots, k_p \leq n\\1\leq l_1,\ldots, l_q \leq n}} \id_{\{S^{(1)}_{k_1}=\ldots=S^{(1)}_{k_p}=S^{(2)}_{l_1}=\ldots=S^{(2)}_{l_q}\}}
\end{equation*}
for all $(p,q)\in \mathbb{N}^2$ and all $n\geq 1$. The mutual intersection local time of two replicas up to time $n$ is $Q_n^{[1,1]}$ and abbreviated as $Q_n$.

\paragraph{Moments estimates of local times.}We collect here a number of useful estimates for the moments of $I_n^{[p]}$ and $Q_n^{[p,q]}$.

 
 \begin{pr}[Mutual intersections of random walks with finite variance] \label{estimate_RW2}
 Suppose $(S_n)_{n\geq0}$ is as in Theorem \ref{quenched_clt3}, Assumption (1).
 \begin{enumerate}
\item If $d\geq5$,
\begin{equation*}
 Q_{\infty} = \sup_{n\geq1} Q_n = \sum_{k,l\geq1} \id_{\{S^{(1)}_k = S^{(2)}_l\}}
\end{equation*}
is a.s finite and
\begin{equation*}
 \forall k\geq1,\quad E^{\otimes 2}[Q_{\infty}^k] < +\infty.
\end{equation*}
\item If $d=4$, for all $k\geq1$ there exists a constant $C=C(k)$ such that
\begin{equation*}
 E^{\otimes 2}[Q_n^k] \leq C (\log n)^k.
\end{equation*}
\item If $d=3$, for all $k\geq1$ there exists a constant $C=C(k)$ such that
\begin{equation*}
 E^{\otimes 2}[Q_n^k] \leq C n^{k/2}.
\end{equation*}
\end{enumerate}
\end{pr}

\begin{pr}[Mutual intersections of random walks with infinite variance] \label{estimate_RW3}
Suppose $(S_n)_{n\geq0}$ is as in Theorem \ref{quenched_clt3}, Assumption (2).
\begin{enumerate}
 \item If $\alpha < d/2$, \begin{equation*}
 Q_{\infty} = \sup_{n\geq1} Q_n = \sum_{k,l\geq1} \id_{\{S^{(1)}_k = S^{(2)}_l\}}
\end{equation*}
is a.s finite and
\begin{equation*}
 \forall k\geq1,\quad E^{\otimes 2}[Q_{\infty}^k] < +\infty.
\end{equation*}
\item If $\alpha = d/2$, for all $k\geq1$ there exists a constant $C=C(k)$ such that
\begin{equation*}
 E^{\otimes 2}[Q_n^k] \leq C (\log n)^k.
\end{equation*}
\item If $d/2 < \alpha < d$, for all $k\geq1$ there exists a constant $C=C(k)$ such that
\begin{equation*}
 E^{\otimes 2}[Q_n^k] \leq C n^{k(2-d/\alpha)} \ll n^k.
\end{equation*}
\end{enumerate} 
\end{pr}


\begin{pr}[Self-intersections and mutual intersections of two-dimensional random walks with finite covariance matrix] \label{estimate_RW4} Suppose that $(S_n)_{n\geq0}$ satisfies the assumptions of Theorem \ref{quenched_clt2}.
\begin{enumerate}
 \item For all $p,k\geq1$, there exists a constant $C = C(p,k)$ such that
 \begin{equation*}
  E[(I_n^{[p]})^k] \leq C n^k (\log n)^{k(p-1)}.
 \end{equation*}
 \item $(I_n/(n\log n))_{n\geq1}$ converges a.s to $\sigma^2$ as $n\rightarrow +\infty$.
 \item For all $k\geq1$,
 \begin{equation*}
  E[I_n^k] \stackrel{n\rightarrow +\infty}{\sim} \sigma^{2k} (n \log n)^k.
 \end{equation*}
 \item For all $k\geq1$ there exists a constant $C=C(k)$ such that
\begin{equation*}
 E^{\otimes 2}[Q_n^k] \leq C n^k.
\end{equation*}
\end{enumerate}
\end{pr}


\begin{proof}[Proof of Proposition \ref{estimate_RW2}]
Items (1), (2) and (3) can be respectively found in  \cite[Lemma 1]{MR1311718}, \cite[Lemma 1]{MR1440255} and \cite[Lemma 6.2.2]{MR2584458}. 
\end{proof}

\begin{proof}[Proof of Proposition \ref{estimate_RW3}]
 We reduce the proof of Items (1), (2) and (3) to the case $k=1$, since it is not too difficult to show (see \cite[Proof of Lemma 5.2]{MR2094445}) that for all $k\geq 2$ there exists a constant $C>0$ such that $$E^{\otimes 2}[Q_n^k] \leq C [E^{\otimes 2}Q_n]^k.$$ The upper bound for  $k=1$ comes from
 \begin{align*}
  E^{\otimes 2}[Q_n] = \sum_{k,l=1}^n P^{\otimes 2}(S_k^{(1)} = S_l^{(2)}) &= \sum_{k,l=1}^n \sum_{x\in \mathbb{Z}^d} P(S_k = x)P(S_l = x) \\
  & = \sum_{k,l=1}^n P(S_{k+l}=0)\\
  &\leq C \sum_{k,l=1}^n (k+l)^{-\frac{d}{\alpha}}\\
  &\leq C \sum_{k=1}^n k^{1-\frac{d}{\alpha}},
 \end{align*}
which is enough to conclude. We have used the Chapman-Kolmogorov equation for symmetric random walks to go from line 1 to line 2, the Local Limit Theorem (see \cite[Theorem 1]{MR0222939}) from line 2 to line 3, and the fact that $d>\alpha$ from line 3 to line 4.
\end{proof}

\begin{proof}[Proof of Proposition \ref{estimate_RW4}]
 Item (2) was proved in \cite[Theorem 1]{MR2290196} and a proof of Item (4) can be found in \cite[Equation (2.10)]{MR2261059}. Since Item (3) is a consequence of Items (1) and (2), we only need to prove Item (1).  We have
\begin{align*}
 &E[(I_n^{[p]})^k]^{1/k}\\  &= E\left[ \left( \sum_{1\leq l_1, \ldots, l_p\leq n} \id_{\{S_{l_1} = S_{l_2} = \ldots = S_{l_p}\}} \right)^k \right]^{1/k}\\
 & \leq C \sum_{l_1=1}^{n} E\left[ \left(  \sum_{l_1 \leq l_2 \leq \ldots \leq l_p\leq n} \id_{\{S_{l_1} = S_{l_2} = \ldots = S_{l_p}\}} \right)^k \right]^{1/k} \mbox{\rm \quad by triangular inequality}\\
 & \leq C \sum_{l_1=1}^{n} E \left[ \left( \sum_{l_1 \leq l_2 \leq \ldots \leq l_p\leq n} \id_{\{S_{l_2- l_1} = \ldots = S_{l_p-l_{p-1}}=0\}} \right)^k \right]^{1/k} \mbox{\rm \quad by stationarity of the increments}\\
 &\leq C n \ E \left( N_n(0) ^{k(p-1)} \right)^{1/k}.
\end{align*}
Since (see the proof of Lemma 2.5 in \cite{MR972774} ) for any $m\geq 1$, 
$$E( N_n(0)^m) \sim C \log(n)^m,$$
Item (1) follows.

\end{proof}

We now focus on the case of a renewal process, that is when $(S_n)_{n\geq 0}$ is as in Theorem \ref{quenched_clt}. In this special setup,
$(S_n)_{n\geq 0}$ is increasing and we obviously have $N_n(i)\in\{0,1\}$ for all $n,i\geq1$. Therefore, $I_n = I_n^{[p]} = n$ for all $n,p\geq 1$. We define for all $p\geq1$,
\begin{equation}\label{defJn}
 J_n^{[p]} = \sum_{i\geq1} \overline{N_n(i)}^p,
\end{equation}
and we abbreviate $J_n^{[2]}$ as $J_n$.
\begin{pr}\label{estimates_renewal}Suppose $(S_n)_{n\geq0}$ is as in Theorem \ref{quenched_clt}. Then
 \begin{enumerate}
  \item $(Q_n/n)_{n\geq1}$ converges a.s to $1/m$.
  \item $(J_n/n)_{n\geq 1}$ converges a.s to $1-1/m$.
  \item For all $p\geq1$, there exists $C=C(p)$ such that
  \begin{equation*}
   J_n^{[p]} \leq Cn.
  \end{equation*}
  \item For all $k\geq1$,
  \begin{equation*}
   E[J_n^k] \stackrel{n \rightarrow +\infty}{\sim} (1-1/m)^k n^k.
  \end{equation*}
  \item There exists a constant $C$ such that
  \begin{equation*}
   E^{\otimes 2}\left[ \left(\frac{Q_n}{n} - \frac{1}{m} \right)^2\right] \leq \frac{C}{n}.
  \end{equation*}
 \end{enumerate}
\end{pr}
 
\begin{rmk}\label{intersection} We will use several times the following identity: if $S^{(1)},\ldots, S^{(k)}$ are $k$ independent copies of  a renewal process $S$, then
 \begin{equation*}
  \sum_{i\geq1} \prod_{j=1}^k N^{(j)}_n(i) = \left| \bigcap_{j=1}^k \{S^{(j)}_1,\ldots,S^{(j)}_n\} \right|.
 \end{equation*}
\end{rmk}

\begin{proof}[Proof of Proposition \ref{estimates_renewal}]
 Let us first prove Item (1). By using Remark \ref{intersection} with $k=2$, we get
\begin{equation*}
 \sum_{i\geq1}N^{(1)}_n(i)N^{(2)}(i) = Q_n := \left| \{S^{(1)}_1,\ldots, S^{(1)}_n\} \cap \{S^{(2)}_1,\ldots, S^{(2)}_n\}\right|.
\end{equation*}
Remark that $S^{\cap}:= S^{(1)}\cap S^{(2)}$ is again a renewal process with interarrival mean (ie. the mean of the increments) equal to
$$\lim_{n\rightarrow +\infty} P(n\in S^{\cap})^{-1} = \lim_{n\rightarrow +\infty} P(n\in S)^{-2} = m^2$$ by the Renewal Theorem (see \cite[Theorem 2.2, Chapter I]{Asmussen}). Therefore we get $$\left|S^{\cap} \cap \{1,\ldots,n\}\right| \stackrel{n\rightarrow +\infty}{\sim} n/m^2.$$ But $$Q_n = S^{\cap}_{\sigma_n}$$ where $$\sigma_n := \min(S^{(1)}_n,S^{(2)}_n),$$ and $\sigma_n \stackrel{n\rightarrow +\infty}{\sim} mn$ almost surely, which gives $$Q_n \stackrel{n\rightarrow +\infty}{\sim} n/m.$$
Let us now prove Item (2). We have
 \begin{equation*}
  \frac{1}{n}\sum_{i\geq1} \barre{N_n(i)}^2 = \frac{1}{n}\sum_{i\geq1} N_n(i)^2 - \frac{2}{n}\sum_{i\geq1} N_n(i)E[N_n(i)] + \frac{1}{n}\sum_{i\geq1} (E[N_n(i)])^2.
 \end{equation*}
First,
\begin{equation*}
 \sum_{i\geq1} N_n(i)^2 = \sum_{i\geq1} N_n(i) =  n.
\end{equation*}
We conclude by remarking that
\begin{align*}
 &\sum_{i\geq1} N_n(i)E[N_n(i)] = E^{(2)} \sum_{i\geq1} N_n(i)N_n^{(2)}(i) = E^{(2)}Q_n \sim n/m,\\
 &\sum_{i\geq1} (E[N_n(i)])^2 = E^{\otimes 2} \sum_{i\geq1} N_n^{(1)}(i)N_n^{(2)}(i) = E^{\otimes 2} Q_n \sim n/m,
\end{align*}
where the asymptotics come from Item (1) and the Bounded Convergence Theorem.
We go now to the proof of Item (3). For all integers $l\geq 2$,
\begin{equation*}
|J_n^{[l]}| \leq \sum_{i\geq1} |\barre{N_n(i)}|^l \leq \sum_{i\geq 1} (N_n(i) + EN_n(i))^l
\leq 2^{l-1} \left(\sum_{i\geq1} N_n(i)^l + \sum_{i\geq1} [EN_n(i)]^l \right).
\end{equation*}
But $\sum_{i\geq1} N_n(i)^l = \sum_{i\geq1} N_n(i) = n$ and (see Remark \ref{intersection})
\begin{equation*}
\sum_{i\geq1} [EN_n(i)]^l = E^{\otimes l} \sum_{i\geq1} N_n^{(1)}(i)\ldots N_n^{(l)}(i) 
= \left| \bigcap_{j=1}^l \{\tau^{(j)}_1,\ldots,\tau^{(j)}_n\} \right|
\leq n.\label{aux2}
\end{equation*}
Item (4) is a straightforward consequence of Items (2) and (3) (by the Bounded Convergence Theorem). We finish with the proof of Item (5).  Let $U = (U_n)_{n\geq1}$ and $V = (V_n)_{n\geq1}$ be the two identically distributed renewal processes defined by 
 \begin{equation}\label{equ1_lemma3}
  S^{\cap}_n = S^{(1)}_{U_n} = S^{(2)}_{V_n}.
 \end{equation}
Then $$Q_n = \min(N_n^U,N_n^U),$$ where 
\begin{align*}
N_n^U &:= U \cap \{1,\ldots,n\},\\N_n^V &:= V \cap \{1,\ldots,n\},
\end{align*}
and so,
\begin{align*}
 E^{\otimes2}\left( \frac{Q_n}{n} - \frac{1}{m} \right)^2 &\leq E^{\otimes 2} \left( \frac{N_n^U}{n} - \frac{1}{m}  \right)^2  + E^{\otimes 2} \left( \frac{N_n^V}{n} - \frac{1}{m}  \right)^2\\
 & = 2 E^{\otimes 2} \left( \frac{N_n^U}{n} - \frac{1}{m}  \right)^2 
\end{align*}
From Equation (\ref{equ1_lemma3}) and the Renewal Theorem, we get $$U_n \sim V_n \stackrel{n\rightarrow +\infty}{\sim} nm$$ almost surely, meaning that $$EU_1 = EV_1 = m.$$ Moreover, $EU_1^2 < +\infty$ since $$U_1 \leq S^{(1)}_{U_1} = S^{\cap}_1$$ and $E[(S^{\cap}_1)^2]<+\infty$ from \cite[Proposition 3.2]{MR1395609}. Then we get
\begin{equation*}
 E^{\otimes 2}\left[ \left( \frac{N_n^U}{n} - \frac{1}{m}  \right)^2 \right]  \leq 2 \frac{\var{N_n^U}}{n^2} + 2\left( \frac{EN_n^U}{n} - \frac{1}{m}\right)^2 = \mathcal{O}\left(\frac{1}{n}\right)
\end{equation*}
by using \cite[Propositions 6.1 and 6.3, Chapter V]{Asmussen}.
\end{proof}

\section{Proof of the results}

\subsection{Proof of Theorem \ref{quenched_clt}}

We suppose $S$ is as in Theorem \ref{quenched_clt}. For all $k\geq1$, we denote by $m_n(k)$ the moments $$m_n(k) :=  E[(\barre{Z_n}/\sqrt{n})^k].$$
Let us recall that if $Z$ is a standard Gaussian random variable, then for all $k\geq1$, $E[Z^{2k-1}]=0$ and $E[Z^{2k}] = (2k-1)!! := (2k-1)\times(2k-3)\times\ldots\times3\times 1$.

 \subsubsection{Control of the moments}
 
 In the following proof we denote by c.f. a combinatorial factor whose precise value is irrelevant and $c_k$ will be a constant only depending on $k$ which may change from line to line.
 
 \begin{lem}[Convergence of quenched averaged moments]\label{lemma2} For all $k\geq1$,
  \begin{align}
  & \mathbb{E}[m_n(2k)] \stackrel{n\rightarrow +\infty}{\longrightarrow} (1-1/m)^k (2k-1)!!\label{even}\\
  &  \mathbb{E}[m_n(2k-1)] =0 \label{odd}
  \end{align}
 \end{lem}
 
 \begin{proof}
 Equation (\ref{odd}) is a direct consequence of Assumption (\ref{asspt1}). Let $k\geq1$ and let us prove Equation (\ref{even}). We have
 \begin{eqnarray*}
 \mathbb{E}[m_n(2k)] &= &\mathbb{E}E[(\barre{Z_n}/\sqrt{n})^{2k}]\\
 &=& \frac{1}{n^k} \sum_{i_1,\ldots, i_{2k}\geq 1} \mathbb{E}(\omega_{i_1}\ldots \omega_{i_{2k}}) E\left[ \barre{N_n(i_1)}\ldots \barre{N_n(i_{2k})}\right]\\
 &=& \frac{1}{n^k} \sum_{j=1}^k C_n(j)
 \end{eqnarray*}
 where
 $$
 C_n(j) := \mathrm{c.f.} \sum_{\substack{i_1,\ldots ,i_j \geq 1 \\ 
 p \neq q \Rightarrow i_p \neq i_q\\ l_1 + \ldots + l_j = 2k;\  l_j \in 2\mathbb{N}^*}} \left[ \prod_{m=1}^j \mathbb{E}[\omega_{i_m}^{l_m}] \right] E\left[\prod_{m=1}^j \barre{N_n(i_m)}^{l_m} \right]
$$
The above sum is restricted to even $l_j$'s because of Assumption (\ref{asspt1}).
For any $ j\leq k-1$, we have
 \begin{eqnarray*}
 C_n(j) &\leq & c_k \sum_{\substack{i_1,\ldots ,i_j \geq 1 \\ l_1 + \ldots + l_j = 2k\\ l_j \in 2\mathbb{N}^*}} \left[ \prod_{m=1}^j \mathbb{E}[\omega_{i_m}^{l_m}] \right] E\left[\prod_{m=1}^j \barre{N_n(i_m)}^{l_m} \right]\\
 &\leq & c_k \sum_{\substack{l_1 + \ldots + l_j = 2k\\ l_j \in 2\mathbb{N}^*}} E\left[ \left( \sum_{i\geq1} \barre{N_n(i)}^{l_1} \right) \ldots \left( \sum_{i\geq1} \barre{N_n(i)}^{l_j} \right) \right]\\
 & = & c_k \sum_{\substack{l_1 + \ldots + l_j = 2k\\ l_j \in 2\mathbb{N}^*}} E[J_n^{[l_1]}\ldots J_n^{[l_j]}]\\
 & \leq & c_k n^j \mbox{\rm \quad from Item (3) of Proposition \ref{estimates_renewal}.}
 \end{eqnarray*}
Therefore,
 \begin{equation}\label{moment1}
 \mathbb{E}[m_n(2k)] = \frac{C_n(k)}{n^k} + \mathcal{O}\left(\frac{1}{n}\right).
 \end{equation}
Let us now compute $C_n(k)$. We have:
\begin{equation*}
C_n(k) = (2k-1)!! \sum_{\substack{i_1,\ldots, i_k \geq 1\\ \mathrm{all\,distinct}}} \mathbb{E}(\omega_{i_1}^2)\ldots \mathbb{E}(\omega_{i_k}^2)E\left[\barre{N_n(i_1)}^2 \ldots \barre{N_n(i_k)}^2\right],
\end{equation*}
$(2k-1)!!$ being the number of pairings of $2k$ elements. From Assumption (\ref{asspt4}) we get
\begin{align*}
C_n(k) &= (2k-1)!! \sum_{\substack{i_1,\ldots, i_k \geq 1\\ \mathrm{all\,distinct}}} E\left[\barre{N_n(i_1)}^2 \ldots \barre{N_n(i_k)}^2\right]\\
&= (2k-1)!! E[J_n^k] \\&\hspace{1cm}- (2k-1)!! \sum_{\substack{i_1,\ldots, i_k \geq 1\\ i_p =i_q {\rm for\,at\,least}\\ {\rm one\,couple } (p,q)}} E\left[\barre{N_n(i_1)}^2 \ldots \barre{N_n(i_k)}^2\right].
\end{align*}
With the same argument as in the case $j\leq k-1$, the second term is in $\mathcal{O}(n^{k-1})$. From Item (4) of Proposition \ref{estimates_renewal} we then get
\begin{equation}\label{moment2}
\frac{C_n(k)}{n^k} \rightarrow \left(1- \frac{1}{m}\right)^k \times (2k-1)!!.
\end{equation}
With Equations (\ref{moment1}) and (\ref{moment2}) we conclude the proof.
\end{proof}

\begin{lem}[Control of the variance of moments]\label{lemma4}
For all $k\geq1$,
\begin{enumerate}
 \item $\bvar[m_n(2k)] = \mathcal{O}(1/n)$
 \item $\bvar[m_n(2k+1)] = \mathcal{O}(1/\sqrt{n})$
\end{enumerate}
\end{lem}

\begin{proof}
Let us start with the variance of even moments. Let $k\geq1$. From the proof of Lemma \ref{lemma2} we get on the one hand
\begin{align*}
[\mathbb{E}m_n(2k)]^2 &= \left( (2k-1)!! E\left[\left( \frac{\sum_{i\geq1} \barre{N_n(i)}^2}{n}\right)^k \right] \right)^2 + \mathcal{O}(1/n)\\
&=\left( (2k-1)!!  E[(J_n/n)^k]\right)^2 + \mathcal{O}(1/n).
\end{align*}
On the other hand,
\begin{align}
&\mathbb{E}[m_n(2k)^2] = \mathbb{E}\Big[E[(\barre{Z_n}/\sqrt{n})^{2k}]^2 \Big]\nonumber \\
& = \frac{1}{n^{2k}}\mathbb{E}\left[\left( \sum_{i_1,\ldots,i_{2k}\geq 1} \omega_{i_1}\ldots \omega_{i_{2k}} E[\barre{N_n(i_1)}\ldots\barre{N_n(i_{2k})}] \right)^2\right]\nonumber \\
& = \frac{1}{n^{2k}} \sum_{\substack{i_1,\ldots , i_{2k}\geq 1\\ j_1,\ldots, j_{2k}\geq 1}} \mathbb{E}[\omega_{i_1}\ldots \omega_{i_{2k}}\omega_{j_1}\ldots \omega_{j_{2k}}]E\left[\barre{N_n(i_1)}\ldots\barre{N_n(i_{2k})}\right]\\&\hspace{6.5cm}\times
E\left[\barre{N_n(j_1)}\ldots\barre{N_n(j_{2k})}\right]\label{Sum}\\
& =: \Sigma_1(n) + \Sigma_2(n) + \Sigma_3(n).\label{sigma}
\end{align}
where $\Sigma_1(n)$ corresponds to pairings of the set $\{i_1,\ldots, i_{2k},j_1,\ldots,j_{2k}\}$ which associate an index with another index of the same group (the two groups of indices being $\{i_1,\ldots,i_k\}$ and $\{j_1,\ldots,j_k\}$). It is equal to
\begin{align*}
& \frac{(2k-1)!!^2}{n^{2k}}\mathbb{E}[\omega_1^2]^{2k} \sum_{\substack{i_1,\ldots , i_{k}\geq 1\\ j_1,\ldots, j_{k}\geq 1\\ \rm{all\,distinct}} } E\left[\barre{N_n(i_1)}^2\ldots \barre{N_n(i_k)}^2\right] 
E\left[\barre{N_n(j_1)}^2\ldots \barre{N_n(j_k)}^2\right]\\
&=  - \frac{(2k-1)!!^2}{n^{2k}}  \sum_{\substack{i_1,\ldots, i_k \geq 1\\j_1,\ldots, j_k \geq 1\\ {\rm at\,least\,two}\\ {\rm \,indexes\,equal}}} E^{\otimes 2}\left[ \barre{N^{(1)}_n(i_1)}^2 \ldots \barre{N^{(1)}_n(i_k)}^2 \barre{N^{(2)}_n(j_1)}^2\ldots \barre{N^{(2)}_n(j_k)}^2\right]\\& \hspace{8cm}+ \frac{(2k-1)!!^2}{n^{2k}} (EJ_n^k)^2.
\end{align*}
Since
\begin{align}
\left| \sum_{i\geq1} \barre{N^{(1)}_n(i)}^k \barre{N^{(2)}_n(i)}^l \right| &
\leq \sqrt{\sum_{i\geq1} \barre{N^{(1)}_n(i)}^{2k}} \sqrt{\sum_{i\geq 1} \barre{N^{(2)}_n(i)}^{2l}}\label{e1} \\
& =\sqrt{J_n^{[2k](1)}J_n^{[2l](2)}} \label{e2}\\&\leq c_{k,l} \times n, \label{e3}
\end{align}
from item (3) of Proposition \ref{estimates_renewal},  the first term in the line above is in $\mathcal{O}(1/n)$. The term $\Sigma_2(n)$ corresponds to mixed pairings, that is pairings for which an index of the set $\{i_1,\ldots, i_{2k}\}$ is paired with an index of the set $\{j_1,\ldots, j_{2k}\}$. More precisely we have:
\begin{equation*}
\Sigma_2(n) = \frac{1}{n^{2k}} \sum_{\substack{2\leq p \leq 2k\\ p\,{\rm even }}} A(p)
\end{equation*}
where $A(p)$ is equal to
\begin{align}
& {\rm c.f.} \sum_{\substack{ k_1,\ldots, k_p \geq1 \\ i_1,\ldots, i_{k-p/2}\geq 1 \\j_1,\ldots, j_{k-p/2}\geq 1\\ \rm{all\, distinct} }} E^{\otimes 2}\left[\barre{N^{(1)}_n(i_1)}^2 \ldots \barre{N^{(1)}_n(i_{k-p/2})}^2  \barre{N^{(2)}_n(j_1)}^2 \ldots \barre{N^{(2)}_n(j_{k-p/2})}^2 \right. \nonumber\\
&\hspace{5.5cm} \left. \times \barre{N^{(1)}_n(k_1)}\barre{N^{(2)}_n(k_1)}\ldots \barre{N^{(1)}_n(k_p)}\barre{N^{(2)}_n(k_p)}\right] \nonumber \\
&\leq c_k E^{\otimes 2}\left[\left(\sum_{i\geq 1} \barre{N^{(1)}_n(i)}^2 \right)^{k-p/2} \left(\sum_{i\geq 1} \barre{N^{(2)}_n(i)}^2 \right)^{k-p/2} \left( \sum_{i\geq 1} \barre{N^{(1)}_n(i)} \barre{N^{(2)}_n(i)} \right)^p\right]\nonumber\\
&\hspace{8cm} + \mathcal{O}(n^{2k-1})\nonumber \\
& := c_k E^{\otimes 2}\left([J_n^{(1)}]^{k-p/2}[J_n^{(2)}]^{k-p/2} \left( \sum_{i\geq 1} \barre{N^{(1)}_n(i)} \barre{N^{(2)}_n(i)} \right)^p\right) + \mathcal{O}(n^{2k-1}) \nonumber
\end{align} 
Since
$$
\sum_{i\geq 1} \barre{N^{(1)}_n(i)} \barre{N^{(2)}_n(i)} = Q_n - E^{(1)}Q_n - E^{(2)}Q_n + E^{\otimes 2} Q_n,
$$
we have
\begin{align}
&\left(\sum_{i\geq 1} \barre{N^{(1)}_n(i)} \barre{N^{(2)}_n(i)} \right)^p = \left[ \left(\sum_{i\geq 1} \barre{N^{(1)}_n(i)} \barre{N^{(2)}_n(i)} \right)^2\right]^{p/2} \\&= (4n)^p \left[ \left(\frac{1}{4n}\sum_{i\geq 1} \barre{N^{(1)}_n(i)} \barre{N^{(2)}_n(i)} \right)^2\right]^{p/2}\nonumber \\
&\leq c_k n^p \left(\frac{1}{n}\sum_{i\geq 1} \barre{N^{(1)}_n(i)} \barre{N^{(2)}_n(i)} \right)^2 \label{aux3}\\
&\leq c_k n^p \left[\left(\frac{Q_n}{n} - \frac{1}{m}\right)^2 + \left(\frac{E^{(1)}Q_n}{n} - \frac{1}{m}\right)^2 + \left(\frac{E^{(2)}Q_n}{n} - \frac{1}{m}\right)^2 \right.\nonumber\\& \hspace{7cm}\left.+ \left(\frac{E^{\otimes 2}Q_n}{n} - \frac{1}{m}\right)^2 \right]\label{aux4}.
\end{align}
Using $[EX]^2 \leq E[X^2]$ we get
\begin{align*}
&\left[ \sum_{i\geq 1} \barre{N^{(1)}_n(i)} \barre{N^{(2)}_n(i)} \right]^p \leq c_k n^p \left[\left(\frac{Q_n}{n} - \frac{1}{m}\right)^2 + E^{(1)}\left(\frac{Q_n}{n} - \frac{1}{m}\right)^2 +\ldots \right. \\  &\hspace{5cm} \left.+ E^{(2)}\left(\frac{Q_n}{n} - \frac{1}{m}\right)^2 + E^{\otimes 2}\left(\frac{Q_n}{n} - \frac{1}{m}\right)^2 \right].
\end{align*}
Since 
\begin{equation*}
[J_n^{(1)}]^{k-p/2}[J_n^{(2)}]^{k-p/2} \leq c_k n^{2k-p},
\end{equation*}
from Item (3) of Proposition \ref{estimates_renewal}, finally using Item (5) of Proposition \ref{estimates_renewal}, we get
\begin{equation*}
A(p) \leq c_k n^{2k} E^{\otimes 2}\left(\frac{Q_n}{n} - \frac{1}{m} \right)^2 = \mathcal{O}(n^{2k-1}),
\end{equation*}
so that:
\begin{equation*}
\Sigma_2(n) = \mathcal{O}(1/n).
\end{equation*}
As for $\Sigma_3(n)$, it is the part of the sum in Equation (\ref{Sum}) when at least four of the indices $i_1,\ldots,i_{2k},j_1,\ldots, j_{2k}$ are equal. 
By combining item (3) of Proposition \ref{estimates_renewal} and equations (\ref{e1}), (\ref{e2}) and (\ref{e3}), we have
\begin{equation*}
\Sigma_3(n) = \mathcal{O}(1/n),
\end{equation*}
which ends the proof for even moments. Let us now consider the variance of odd moments. Let $k\geq1$. On one side $\mathbb{E}[m_n(2k+1)] = 0$. Therefore,
\begin{align*}
&\bvar(m_n(2k+1))= \mathbb{E}(m_n(2k+1)^2) = \mathbb{E}\left[ E[(\barre{Z_n}/\sqrt{n})^{2k+1}]^2\right]\\
&= \frac{1}{n^{2k+1}} \sum_{\substack{i_1,\ldots,i_{2k+1}\geq 1\\j_1,\ldots,j_{2k+1}\geq 1}} \mathbb{E}[\omega_{i_1}\ldots \omega_{i_{2k+1}}\omega_{j_1}\ldots \omega_{j_{2k+1}}]\\&\hspace{4cm}\times E[\barre{N_n(i_1)}\ldots\barre{N_n(i_{2k+1})}]E[\barre{N_n(j_1)}\ldots\barre{N_n(j_{2k+1})}]\\
&= \Sigma_1(n) + \Sigma_2(n),
\end{align*}
where $\Sigma_1(n)$ is the part corresponding to pairings of the set $\{i_1,\ldots,i_{2k+1},j_1,\ldots,j_{2k+1}\}$ and $\Sigma_2(n)$ is the other part. With the same argument as above we get 
\begin{equation*}
\Sigma_2(n) = \mathcal{O}(1/n)
\end{equation*}
and
\begin{equation*}
\Sigma_1(n) = \frac{1}{n^{2k+1}} \sum_{\substack{1\leq p\leq 2k+1\\ p \,{\rm odd}}} A(p)
\end{equation*}
where $A(p)$ is equal to
\begin{align*}
&{\rm c.f.} \sum_{\substack{k_1,\ldots,k_p\geq1\\ i_1,\ldots,i_{k+(1-p)/2}\geq1\\ j_1,\ldots,j_{k+(1-p)/2}\geq1\\ \rm{all distinct}}}  E^{\otimes 2}\left[\barre{N^{(1)}_n(i_1)}^2 \ldots \barre{N^{(1)}_n(i_{k+(1-p)/2})}^2 \barre{N^{(2)}_n(j_1)}^2 \ldots \barre{N^{(2)}_n(j_{k+(1-p)/2})}^2\right. \nonumber\\&\hspace{6cm} \left. \times \barre{N^{(1)}_n(k_1)}\barre{N^{(2)}_n(k_1)}\ldots \barre{N^{(1)}_n(k_p)}\barre{N^{(2)}_n(k_p)}\right]\\
&= {\rm c.f.} E^{\otimes 2}\left[\left(\sum_{i\geq 1} \barre{N^{(1)}_n(i)}^2 \right)^{k+\frac{1-p}{2}} \left(\sum_{i\geq 1} \barre{N^{(2)}_n(i)}^2 \right)^{k+\frac{1-p}{2}} \left( \sum_{i\geq 1} \barre{N^{(1)}_n(i)} \barre{N^{(2)}_n(i)} \right)^p\right] \\&\hspace{9cm}+ \mathcal{O}(n^{2k})\\
&= {\rm c.f.} E^{\otimes 2} \left[ (J_n^{(1)})^{k+(1-p)/2}(J_n^{(2)})^{k+(1-p)/2}\left( \sum_{i\geq 1} \barre{N^{(1)}_n(i)} \barre{N^{(2)}_n(i)} \right)^p \right] + \mathcal{O}(n^{2k}).
\end{align*}
Then
\begin{align*}
|A(p)| &\leq c_k n^{2k+1-p} E^{\otimes 2}\left|\sum_{i\geq 1} \barre{N^{(1)}_n(i)} \barre{N^{(2)}_n(i)} \right|^p\\
&\leq c_k n^{2k+1} E^{\otimes 2}\left|\frac{1}{n}\sum_{i\geq 1} \barre{N^{(1)}_n(i)} \barre{N^{(2)}_n(i)} \right|\\
&\leq c_k n^{2k+1} \sqrt{E^{\otimes 2} \left( \frac{1}{n}\sum_{i\geq 1} \barre{N^{(1)}_n(i)} \barre{N^{(2)}_n(i)}  \right)^2}\\
&\leq c_k n^{2k+1/2}.
\end{align*}
The last line is obtained by the same computations as in Equations (\ref{aux3}), (\ref{aux4}) and Item (5) of Proposition \ref{estimates_renewal}. Finally, $\Sigma_2(n) = \mathcal{O}(1/\sqrt{n})$, which ends the proof.
\end{proof}

 \subsubsection{Conclusion}
 
\begin{proof}[Proof of Theorem \ref{quenched_clt}]
We now prove that $\omega$-almost surely, the moments of $\barre{Z_n}/\sqrt{n}$ with respect to $P$ converge to the moments of $Z$, Gaussian random variable with zero mean and variance $1-1/m$. Let $k\geq1$ and $t_m := [m^{\nu}]$ with $\nu > 2$. First we have
\begin{equation*}
|m_{t_m}(k) - E(Z^k)| \leq |m_{t_m}(k) - \mathbb{E}m_{t_m}(k)| + |\mathbb{E}m_{t_m}(k) - E(Z^k)|.
\end{equation*}
The first term tends to zero $\omega$ almost surely because of Lemma \ref{lemma4} and Borel-Cantelli Lemma, and the second term tends to zero because of Lemma \ref{lemma2}. We are left with the control of $|m_n(k) - m_{t_m}(k)|$ for $t_m < n \leq t_{m+1}$. Let 
\begin{equation*}
D_m = \sup_{t_m < n \leq t_{m+1}} |E[\barre{Z_n}^k] - E[\barre{Z_{t_m}}^k]|.
\end{equation*}
We have for all $t_m < n \leq t_{m+1}$
\begin{align*}
|m_n(k) - m_{t_m}(k)| &\leq t_m^{-k/2}|E[\barre{Z_n}^k] - E[\barre{Z_{t_m}}^k]| +  |E[\barre{Z_{t_m}}^k]|\left| \frac{1}{n^{k/2}} - \frac{1}{t_m^{k/2}}\right|\\
&\leq D_m t_m^{-k/2} + |E[\barre{Z_{t_m}}^k]|\left| \frac{1}{t_{m+1}^{k/2}} - \frac{1}{t_m^{k/2}}\right|\\
&\leq D_m t_m^{-k/2} + |E[(\barre{Z_{t_m}}/\sqrt{t_m})^k]| \left|\left(\frac{t_m}{t_{m+1}}\right)^{k/2} -1 \right|.
\end{align*}
In the second term of the last line, the first factor converges $\omega$-almost surely whereas the second factor obviously tends to $0$ from the definition of $t_m$. Therefore, we now only need to prove that the first term tends to zero $\omega$-almost surely. For any even integer $l\geq2$,
\begin{align}
t_m^{-kl/2}\mathbb{E}D_m^l &\leq t_m^{-kl/2} \sum_{n= t_m+1}^{t_{m+1}} \mathbb{E}\left[ |E[\barre{Z_n}^k - \barre{Z_{t_m}}^k]|^l\right]\nonumber \\
&\leq t_m^{-kl/2} \sum_{n= t_m+1}^{t_{m+1}} \mathbb{E}\left[ \left[E(\barre{Z_n}-\barre{Z_{t_m}})(\sum_{i=0}^{k-1}\barre{Z_n}^i \barre{Z_{t_m}}^{k-1-i}  )\right]^l \right]\nonumber \\
&\leq c t_m^{-kl/2} \sum_{n= t_m+1}^{t_{m+1}} \mathbb{E}E (\barre{Z_n}-\barre{Z_{t_m}})^l \sum_{i=0}^{k-1} \barre{Z_n}^{il} \barre{Z_{t_m}}^{(k-1-i)l}\nonumber\\
&\leq c t_m^{-kl/2}  \sum_{n= t_m+1}^{t_{m+1}}  \left(E\mathbb{E}(\barre{Z_n}-\barre{Z_{t_m}})^{2l} \right)^{1/2} \left(\sum_{i=0}^{k-1} E\mathbb{E}\barre{Z_n}^{2il} \barre{Z_{t_m}}^{2(k-1-i)l}  \right)^{1/2}\label{step1} \\
&\leq c t_m^{-kl/2} \sum_{n= t_m+1}^{t_{m+1}} (n-t_m)^{l/2} \left(\sum_{i=0}^{k-1} n^{il} t_m^{l(k-1-i)}\right)^{1/2}\label{step2}\\
&\leq c t_m^{-kl/2} \left( \sum_{n= t_m+1}^{t_{m+1}} (n-t_m)^{l/2} \right) t_m^{(k-1)(l/2)}\nonumber \\
&\leq c t_m^{-l/2} (t_{m+1}-t_m)^{1+l/2}\nonumber\\
&\leq C m^{\nu-1-l/2}.\nonumber
\end{align}
To go from (\ref{step1}) to (\ref{step2}), we use the fact that $\mathbb{E}E[\overline{Z_n}^k] \leq C n^{k/2}$ (Lemma \ref{lemma2}) plus the fact that $$\mathbb{E}E[(\barre{Z_{n}}- \barre{Z_{t_m}})^{2k}]\leq C (n- t_m)^k,$$ which comes as a slight modification of what we did in the proof of Lemma \ref{lemma2}. It is then enough to choose $l>2\nu$ to apply Borel-Cantelli Lemma and obtain $\omega$ almost sure convergence to zero of $D_m t_m^{-k/2}$.
\end{proof}

\subsection{Proof of Theorem \ref{quenched_clt2}}

We assume $S$ is as in Theorem \ref{quenched_clt2}. For all $k\geq1$, we denote by $m_n(k)$ the moments $$m_n(k) :=  E[(Z_n/\sqrt{n \log n})^k].$$
The proof is very similar to the proof of Theorem \ref{quenched_clt}. The main difference is that the process $(Z_n)_{n\geq0}$ doesn't need to be recentered. As a consequence, the $J_n$'s defined in Equation (\ref{defJn}) will be replaced by the $I_n$'s, and the $\overline{N_n(i)}$'s by the $N_n(i)$'s.

 \subsubsection{Control of the moments}
 
 \begin{lem}[Convergence of quenched averaged moments]\label{lemma2RW} For all $k\geq1$,
  \begin{align}
   &\mathbb{E}[m_n(2k)] \stackrel{n\rightarrow +\infty}{\longrightarrow} \sigma^{2k} (2k-1)!!\label{evenRW}\\
   &\mathbb{E}[m_n(2k-1)] =0 \label{oddRW}
  \end{align}
 \end{lem}
 
 \begin{proof}
 Equation (\ref{oddRW}) is a direct consequence of Assumption (\ref{asspt1}). Let $k\geq1$ and let us prove Equation (\ref{evenRW}). Adapting the proof of Lemma \ref{lemma2}, we have:
 \begin{align*}
 \mathbb{E}[m_n(2k)] &= \mathbb{E}E[(Z_n/\sqrt{n\log n})^{2k}]\\
 &=\frac{1}{(n\log n)^k} \sum_{i_1,\ldots, i_{2k}\geq 1} \mathbb{E}(\omega_{i_1}\ldots \omega_{i_{2k}}) E\left[ N_n(i_1)\ldots N_n(i_k)\right]\\
 &= \frac{1}{(n\log n)^k} \sum_{j=1}^k C_n(j)
 \end{align*}
 where
 \begin{equation*}
 C_n(j) := \mathrm{c.f.} \sum_{\substack{i_1,\ldots ,i_j \geq 1 \\ p \neq q \Rightarrow i_p \neq i_q\\ l_1 + \ldots + l_j = 2k\\ l_j \in 2\mathbb{N}^*}} \left[ \prod_{m=1}^j \mathbb{E}[\omega_{i_m}^{l_m}] \right] E\left[\prod_{m=1}^j N_n(i_m)^{l_m} \right]
 \end{equation*}
 For all $j\leq k-1$, we can prove as in the previous section, using Item (1) of Proposition \ref{estimate_RW4} that:
 \begin{equation*}
 C_n(j) \leq c_k n^j (\log n)^{2k-j}.
 \end{equation*}
Moreover,
\begin{equation*}
C_n(k) = (2k-1)!! \sum_{\substack{i_1,\ldots, i_k \geq 1\\ \mathrm{all\,distinct}}} \mathbb{E}(\omega_{i_1}^2)\ldots \mathbb{E}(\omega_{i_k}^2)E[N_n(i_1)^2 \ldots N_n(i_k)^2],
\end{equation*}
so from Item (3) of Proposition \ref{estimate_RW4}, we get
\begin{equation}\label{moment2RW}
\frac{C_n(k)}{(n \log n)^k} \rightarrow \sigma^{2k} \times (2k-1)!!,
\end{equation}
which concludes the proof.
\end{proof}

\begin{lem}[Control of the variance of moments]\label{lemma4RW} 
We have the following:
 \begin{align*}
  &\bvar[m_n(2k)] = \mathcal{O}(1/(\log n)^2),\\
 &\bvar[m_n(2k+1)] = \mathcal{O}(1/\log n).
  \end{align*}
\end{lem}

\begin{proof}
We only adapt the proof given in the previous section. Let us start with the variance of even moments. Let $k\geq1$. From the proof of Lemma \ref{lemma2RW} we can get on the one hand
\begin{equation*}
[\mathbb{E}m_n(2k)]^2 =\left( (2k-1)!!  E[(I_n/(n \log n))^k]\right)^2 + \mathcal{O}(\log n / n),
\end{equation*}
and on the other hand,
\begin{align*}
&\mathbb{E}[m_n(2k)^2] = \Sigma_1(n) + \Sigma_2(n) + \Sigma_3(n),
\end{align*}
where the $\Sigma_i(n)$ ($i\in\{1,2,3\}$) are the same as those appearing in Equation (\ref{sigma}), except that the recentered local times $\overline{N_n(i)}$'s must be replaced by the local times $N_n(i)$'s, and the $J_n$'s by the $I_n$'s. With these modifications, the proof of Lemma \ref{lemma4} can be reproduced, and thanks to Item (1) of Proposition \ref{estimate_RW4}, we get
\begin{equation*}
\Sigma_1(n) = \frac{(2k-1)!!^2}{(n \log n)^{2k}} (EI_n^k)^2 + \mathcal{O}(\log n /n).
\end{equation*}
From the same item we also have
\begin{equation*}
\Sigma_3(n) = \mathcal{O}(\log n / n).
\end{equation*}
Finally,
\begin{equation*}
\Sigma_2(n) = \frac{1}{(n \log n)^{2k}} \sum_{\substack{2\leq p \leq 2k\\ p \,{\rm even }}} A(p)
\end{equation*}
where 
\begin{equation*}
A(p) \leq c_k E^{\otimes 2}\left([I_n^{(1)}]^{k-p/2}[I_n^{(2)}]^{k-p/2} Q_n^p\right) + \mathcal{O}(n^{2k-1} (\log n)^{2k+1}).
\end{equation*}
Using multidimensional H\"older inequality, we obtain
\begin{align*}
 &E^{\otimes 2}\left([I_n^{(1)}]^{k-p/2}[I_n^{(2)}]^{k-p/2} Q_n^p\right)\\
 &\leq \left[E^{\otimes2} (I_n^{(1)})^{3k-3p/2} \right]^{1/3} \left[E^{\otimes2} (I_n^{(2)})^{3k-3p/2} \right]^{1/3} [E^{\otimes2} Q_n^{3p}]^{1/3}\\
 &\leq \left[E I_n^{3k-3p/2} \right]^{2/3} [E^{\otimes2} Q_n^{3p}]^{1/3} \\
 & \leq C n^{2k}  (\log n)^{2k-p}\quad \mbox{\rm from Items (1) and (4) of Proposition \ref{estimate_RW4}.}
\end{align*}
This ends the proof for even moments.

Let us now consider the variance of odd moments. Let $k\geq1$. Again, by adapting the proof of Lemma \ref{lemma4}, we get
\begin{equation*}
\bvar [m_n(2k+1)] = \frac{1}{(n \log n)^{2k+1}} \sum_{\substack{1\leq p\leq 2k+1\\ p \rm{\, odd}}} A(p) + \mathcal{O}(\log n /n).
\end{equation*}
where $A(p)$ is equal to
\begin{align*}
&{\rm c.f.} \sum_{\substack{k_1,\ldots,k_p\geq1\\ i_1,\ldots,i_{k+(1-p)/2}\geq1\\ j_1,\ldots,j_{k+(1-p)/2}\geq1\\ \rm{all distinct}}}  E^{\otimes 2}\left[N^{(1)}_n(i_1)^2 \ldots N^{(1)}_n(i_{k+(1-p)/2})^2 N^{(2)}_n(j_1)^2 \ldots N^{(2)}_n(j_{k+(1-p)/2})^2 \right. \\ &\hspace{6cm} \left.\times N^{(1)}_n(k_1)N^{(2)}_n(k_1)\ldots N^{(1)}_n(k_p)N^{(2)}_n(k_p)\right]\\
&= {\rm c.f.} E^{\otimes 2}\left[\left(\sum_{i\geq 1} N^{(1)}_n(i)^2 \right)^{k+(1-p)/2} \left(\sum_{i\geq 1} N^{(2)}_n(i)^2 \right)^{k+(1-p)/2} \left( \sum_{i\geq 1} N^{(1)}_n(i) N^{(2)}_n(i) \right)^p\right]\\& \hspace{9cm} + \mathcal{O}(n^{2k} (\log n)^{2k+2})\\
& = {\rm c.f.} E^{\otimes 2}[(I_n^{(1)})^{k+(1-p)/2}(I_n^{(2)})^{k+(1-p)/2}Q_n^p] + \mathcal{O}(n^{2k} (\log n)^{2k+2})
\end{align*}
We then use H\"older inequality to obtain
\begin{align*}
 &E^{\otimes 2}\left([I_n^{(1)}]^{k+(1-p)/2}[I_n^{(2)}]^{k+(1-p)/2} Q_n^p\right)\\
 &\leq \left[E^{\otimes2} (I_n^{(1)})^{3k+3(1-p)/2} \right]^{1/3} \left[E^{\otimes2} (I_n^{(2)})^{3k+3(1-p)/2} \right]^{1/3} [E^{\otimes2} Q_n^{3p}]^{1/3}\\
 &\leq \left[E I_n^{3k+3(1-p)/2} \right]^{2/3} [E^{\otimes2} Q_n^{3p}]^{1/3} \\
 & \leq C n^{2k+1} (\log n)^{2k+1-p} \quad \mbox{\rm from Items (1) and (4) of Proposition \ref{estimate_RW4} },
\end{align*}
which ends the proof.
\end{proof}
Let $t_m=[\exp(m^{1+\nu})]$ for some $\nu>0$.
From Borel-Cantelli lemma,  $\omega$-almost surely, the moments of $Z_{t_m}/\sqrt{t_m \log t_m}$ with respect to $P$ converge to the moments of $Z$, Gaussian random variable with zero mean and variance $\sigma^2$. Theorem \ref{quenched_clt2} follows from the classical moment theorem.

\subsection{Proof of Theorem \ref{quenched_clt3}}
The random process $\xi:=(\xi_k)_{k\geq 0}$ defined as $\xi_k:= (\omega_{S_k+x} )_{x\in\mathbb{Z}^d} , k\geq 0$  is a Markov chain on the state space $X:=(\mathbb{R})^{\mathbb{Z}^d}$ with transition 
operator defined for any bounded measurable function $f$ as 
$$ P f (\omega) = E [ f (   (\omega_{X_1+x} )_{x\in\mathbb{Z}^d}) ].$$
The Markov chain $\xi$ is stationary and ergodic (from Kakutani random ergodic theorem \cite{KAKU}). The stationary law $\mu$ is given by the product law of the random
variables $w_x, x\in\mathbb{Z}^d$.
A direct application of the main theorem in \cite{DL} gives us the result. Indeed, choose $f$ as the projection on the zero component i.e.
$f: X\rightarrow \mathbb{R}; \omega\rightarrow \omega_0$. The random scenery being assumed centered and square integrable, it implies that $\int f(x) d\mu(x)=0$ and $f\in L^2(\mu)$.
Then,
\begin{eqnarray*}
\left|\left| \sum_{k=1}^n P^k f \right|\right|_{2,\mu}^2 & =& \left|\left| \sum_{k=1}^n E [ f (   (\omega_{S_k+x} )_{x\in\mathbb{Z}^d}) ] \right|\right|_{2,\mu}^2\\
&=& \mathbb{E} \left[  \left( \sum_{k=1}^n E[ w_{S_k} ]\right)^2 \right] \\
&=&\mathbb{E} \left[ \left( \sum_{x\in\mathbb{Z}^d} \omega_x N_n(x) \right)^2 \right]\\
&=& \sum_{x,y\in\mathbb{Z}^d}  \mathbb{E} (w_x w_y) E^{\otimes 2} [ N_n^{(1)} (x) N_n^{(2)}(y)]  \\
&=& \sum_{x\in\mathbb{Z}^d} E^{\otimes 2}[N_n^{(1)}(x)N_n^{(2)}(x)]\\
&=& E^{\otimes 2} [ Q_n] = \mathcal{O} (n^{\alpha}) 
\end{eqnarray*} 
with $0\leq \alpha <1$ from Propositions \ref{estimate_RW2} and \ref{estimate_RW3}.

\section{Conclusion and further comments}
As stated in the introduction, this paper is a first attempt in proving distributional limit theorems for random walk in {\it quenched} random scenery.
The method of moments we used to prove Theorems \ref{quenched_clt} and \ref{quenched_clt2} requires the existence of the moments of any order of  the scenery. To relax this assumption a new approach should be developed.\\*
The case of the dimension one was not discussed in the paper and is far from being trivial. 
We are only able to prove that if there is such a quenched limit theorem then the limit law necessarily 
depends on the scenery. \\*
Another non trivial question is to extend Theorem \ref{quenched_clt2} in order to get a limit theorem along the full sequence of the integers.

\bibliographystyle{spmpsci}      
\bibliography{references}

\def\polhk\#1{\setbox0=\hbox{\#1}{{\o}oalign{\hidewidth
  {\l}ower1.5ex\hbox{`}\hidewidth\crcr\unhbox0}}}
\begin{thebibliography}{10}
\providecommand{\url}[1]{{#1}}
\providecommand{\urlprefix}{URL }
\expandafter\ifx\csname urlstyle\endcsname\relax
  \providecommand{\doi}[1]{DOI~\discretionary{}{}{}#1}\else
  \providecommand{\doi}{DOI~\discretionary{}{}{}\begingroup
  \urlstyle{rm}\Url}\fi

\bibitem{Asmussen}
Asmussen, S.: Applied probability and queues, \emph{Applications of Mathematics
  (New York)}, vol.~51, second edn.
\newblock Springer-Verlag, New York (2003).
\newblock Stochastic Modelling and Applied Probability

\bibitem{MR2288063}
Asselah, A., Castell, F.: Random walk in random scenery and self-intersection
  local times in dimensions {$d\ge5$}.
\newblock Probab. Theory Related Fields \textbf{138}(1-2), 1--32 (2007)

\bibitem{MR2261059}
Bass, R.F., Chen, X., Rosen, J.: Moderate deviations and laws of the iterated
  logarithm for the renormalized self-intersection local times of planar random
  walks.
\newblock Electron. J. Probab. \textbf{11}, no. 37, 993--1030 (electronic)
  (2006)

\bibitem{MR2353391}
Ben~Arous, G., {\v{C}}ern{\'y}, J.: Scaling limit for trap models on {$\Bbb
  Z^d$}.
\newblock Ann. Probab. \textbf{35}(6), 2356--2384 (2007)

\bibitem{MR972774}
Bolthausen, E.: A central limit theorem for two-dimensional random walks in
  random sceneries.
\newblock Ann. Probab. \textbf{17}(1), 108--115 (1989)

\bibitem{MR543530}
Borodin, A.N.: A limit theorem for sums of independent random variables defined
  on a recurrent random walk.
\newblock Dokl. Akad. Nauk SSSR \textbf{246}(4), 786--787 (1979)

\bibitem{MR535455}
Borodin, A.N.: Limit theorems for sums of independent random variables defined
  on a transient random walk.
\newblock Zap. Nauchn. Sem. Leningrad. Otdel. Mat. Inst. Steklov. (LOMI)
  \textbf{85}, 17--29, 237, 244 (1979).
\newblock Investigations in the theory of probability distributions, IV

\bibitem{MR2060457}
Castell, F.: Moderate deviations for diffusions in a random {G}aussian shear
  flow drift.
\newblock Ann. Inst. H. Poincar\'e Probab. Statist. \textbf{40}(3), 337--366
  (2004)

\bibitem{CGPP}
{Castell}, F., {Guillotin-Plantard}, N., {P\`ene}, F.: Limit theorems for one
  and two-dimensional random walks in random scenery.
\newblock Annales de l'Institut Henri Poincar\'e - Probabilit\'es et
  Statistiques  (To appear)

\bibitem{MR1840830}
Castell, F., Pradeilles, F.: Annealed large deviations for diffusions in a
  random {G}aussian shear flow drift.
\newblock Stochastic Process. Appl. \textbf{94}(2), 171--197 (2001)

\bibitem{MR2290196}
{\v{C}}ern{\'y}, J.: Moments and distribution of the local time of a
  two-dimensional random walk.
\newblock Stochastic Process. Appl. \textbf{117}(2), 262--270 (2007)

\bibitem{MR2094445}
Chen, X.: Exponential asymptotics and law of the iterated logarithm for
  intersection local times of random walks.
\newblock Ann. Probab. \textbf{32}(4), 3248--3300 (2004)

\bibitem{MR2584458}
Chen, X.: Random walk intersections, \emph{Mathematical Surveys and
  Monographs}, vol. 157.
\newblock American Mathematical Society, Providence, RI (2010).
\newblock Large deviations and related topics

\bibitem{MR2571841}
Cohen, S., Dombry, C.: Convergence of dependent walks in a random scenery to
  f{B}m-local time fractional stable motions.
\newblock J. Math. Kyoto Univ. \textbf{49}(2), 267--286 (2009)

\bibitem{MR1700010}
Cs{\'a}ki, E., K{\"o}nig, W., Shi, Z.: An embedding for the {K}esten-{S}pitzer
  random walk in random scenery.
\newblock Stochastic Process. Appl. \textbf{82}(2), 283--292 (1999)

\bibitem{MR688990}
Cs{\'a}ki, E., R{\'e}v{\'e}sz, P.: Strong invariance for local times.
\newblock Z. Wahrsch. Verw. Gebiete \textbf{62}(2), 263--278 (1983)

\bibitem{MR2883216}
Deligiannidis, G., Utev, S.A.: Computation of the asymptotics of the variance
  of the number of self-intersections of stable random walks using the
  {W}iener-{D}arboux theory.
\newblock Sibirsk. Mat. Zh. \textbf{52}(4), 809--822 (2011)

\bibitem{DL}
Derriennic, Y., Lin, M.: The central limit theorem for markov chains started at
  a point.
\newblock Probab. Theory Related Fields \textbf{125}(1), 73--76

\bibitem{MR1395609}
Diaconis, P., Holmes, S., Janson, S., Lalley, S.P., Pemantle, R.: Metrics on
  compositions and coincidences among renewal sequences.
\newblock In: Random discrete structures ({M}inneapolis, {MN}, 1993), \emph{IMA
  Vol. Math. Appl.}, vol.~76, pp. 81--101. Springer, New York (1996)

\bibitem{MR2288269}
Gantert, N., K{\"o}nig, W., Shi, Z.: Annealed deviations of random walk in
  random scenery.
\newblock Ann. Inst. H. Poincar\'e Probab. Statist. \textbf{43}(1), 47--76
  (2007)

\bibitem{MR2779486}
Guillotin-Plantard, N., Prieur, C.: Central limit theorem for sampled sums of
  dependent random variables.
\newblock ESAIM Probab. Stat. \textbf{14}, 299--314 (2010)

\bibitem{MR2744890}
Guillotin-Plantard, N., Prieur, C.: Limit theorem for random walk in weakly
  dependent random scenery.
\newblock Ann. Inst. Henri Poincar\'e Probab. Stat. \textbf{46}(4), 1178--1194
  (2010)

\bibitem{MR2030745}
Guillotin-Plantard, N., Schneider, D.: Limit theorems for sampled dynamical
  systems.
\newblock Stoch. Dyn. \textbf{3}(4), 477--497 (2003)

\bibitem{MR2306188}
den Hollander, F., Steif, J.E.: Random walk in random scenery: a survey of some
  recent results.
\newblock In: Dynamics \& stochastics, \emph{IMS Lecture Notes Monogr. Ser.},
  vol.~48, pp. 53--65. Inst. Math. Statist., Beachwood, OH (2006)

\bibitem{KAKU}
Kakutani, S.: Random ergodic theorems and {M}arkoff processes with a stable
  distribution.
\newblock In: Proceedings of the {S}econd {B}erkeley {S}ymposium on
  {M}athematical {S}tatistics and {P}robability, 1950, pp. 247--261. University
  of California Press, Berkeley and Los Angeles (1951)

\bibitem{MR550121}
Kesten H., S.F.: A limit theorem related to a new class of self-similar
  processes.
\newblock Z. Wahrsch. Verw. Gebiete \textbf{50}(1), 5--25 (1979)

\bibitem{MR1311718}
Khanin, K.M., Mazel, A.E., Shlosman, S.B., Sina{\u\i}, Y.G.: Loop condensation
  effects in the behavior of random walks.
\newblock In: The {D}ynkin {F}estschrift, \emph{Progr. Probab.}, vol.~34, pp.
  167--184. Birkh\"auser Boston, Boston, MA (1994)

\bibitem{MR1624017}
Khoshnevisan, D., Lewis, T.M.: A law of the iterated logarithm for stable
  processes in random scenery.
\newblock Stochastic Process. Appl. \textbf{74}(1), 89--121 (1998)

\bibitem{MR1440255}
Marcus, M.B., Rosen, J.: Laws of the iterated logarithm for intersections of
  random walks on {${\bf Z}^4$}.
\newblock Ann. Inst. H. Poincar\'e Probab. Statist. \textbf{33}(1), 37--63
  (1997)

\bibitem{MR0388547}
Spitzer, F.: Principles of random walks, second edn.
\newblock Springer-Verlag, New York (1976).
\newblock Graduate Texts in Mathematics, Vol. 34

\bibitem{MR0222939}
Stone, C.: On local and ratio limit theorems.
\newblock In: Proc. {F}ifth {B}erkeley {S}ympos. {M}ath. {S}tatist. and
  {P}robability ({B}erkeley, {C}alif., 1965/66), {V}ol. {II}: {C}ontributions
  to {P}robability {T}heory, {P}art 2, pp. 217--224. Univ. California Press,
  Berkeley, Calif. (1967)

\end{thebibliography}
\end{document}